\documentclass{amsart}
\usepackage{amssymb,amsthm, amsmath, amsfonts}
\usepackage{graphics}
\usepackage{hyperref}
\usepackage[all]{xy}
\usepackage{enumerate}
\usepackage[mathscr]{eucal}
\usepackage{bbm}

\theoremstyle{plain}
\newtheorem{theorem}{Theorem}[section]
\newtheorem{lemma}[theorem]{Lemma}
\newtheorem{proposition}[theorem]{Proposition}

\newtheorem{corollary}[theorem]{Corollary}
\numberwithin{equation}{section}

\theoremstyle{definition}

\newtheorem{remark}[theorem]{Remark}

\DeclareMathOperator{\Hom}{Hom}
\DeclareMathOperator{\Tor}{Tor}
\DeclareMathOperator{\hd}{hd}
\DeclareMathOperator{\gd}{gd}
\DeclareMathOperator{\pd}{pd}
\DeclareMathOperator{\td}{td}
\DeclareMathOperator{\id}{id}
\DeclareMathOperator{\fdim}{findim}

\DeclareMathOperator{\Image}{Im}

\newcommand{\mk}{{\mathbbm{k}}}
\newcommand{\C}{{\mathscr{C}}}
\newcommand{\tC}{{\underline{\mathscr{C}}}}
\newcommand{\Z}{{\mathbb{Z}_+}}
\newcommand{\FI}{{\mathscr{FI}}}

\title{Upper bounds of homological invariants of $FI_G$-modules}
\author{Liping Li}
\address{College of Mathematics and Computer Science, Hunan Normal University; Key Laboratory of Performance Computing and Stochastic Information Processing (Hunan Normal University), Ministry of Education; Changsha, Hunan 410081, China.}
\email{lipingli@hunnu.edu.cn.}
\thanks{The author is supported by the National Natural Science Foundation of China 11541002, the Construct Program of the Key Discipline in Hunan Province, and the Start-Up Funds of Hunan Normal University 830122-0037. He also thanks the referee for carefully checking the manuscript and providing many helpful suggestions.}

\begin{document}

\begin{abstract}
In this paper we describe upper bounds for a few homological invariants of $\FI_G$-modules $V$. These upper bounds are expressed in terms of the generating degree and torsion degree, which measure the ``top" and ``socle" of $V$ under actions of non-invertible morphisms in the category respectively.
\end{abstract}

\maketitle

\section{Introduction}

\subsection{Motivation}

To study the representation theory of category $\FI$ and its variations, recently quite a few homological invariant were introduced in literature, including \emph{homological degrees}, \emph{generating degrees}, \emph{torsion degrees}, \emph{regularities}, \emph{derived regularities}, \emph{widths}, \emph{depths}; see for instances \cite{CE, Li, LY, R}. Moreover, people found various upper bounds for these homological invariants; see \cite[Theorem A]{CE} by Church and Ellenberg, \cite[Theorems 1.4 and 1.8]{LY} by the author and Yu, and \cite[Theorems A, C and D]{R} by Ramos.

Although a lot of proceedings have been achieved along this direction, many problems still need to be clarified or solved. Here we list several aspects:
\begin{enumerate}
\item Relationships between various homological invariants need to be established. For example, the two upper bounds of Castelnuovo-Mumford regularity described in \cite[Theorem A]{CE} and \cite[Theorem 1.8]{LY} strongly suggest that there is a close relationship lying between the torsion degree of a finitely generated $\FI$-module and its first two homological degrees. A clear description of this relationship will greatly enhance the understanding of these two upper bounds.
\item Some homological invariants such as derived regularity in \cite[Theorem C]{R}, though have important theoretic meaning, are hard to compute. Therefore, for practical purpose, it would be useful to get their upper bounds in terms of other homological invariants which are easy to obtain.
\item Upper bounds of some homological invariants are still unknown, including the injective dimensions of finitely generate $\FI$-modules over fields of characteristic 0, and torsion degrees of homologies in the finite complex of filtered modules (\cite[Theorem C]{LY}).
\item The proofs of many known results mentioned above rely on a detailed investigation of the combinatorial structure of the category $\FI$. A conceptual approach is desirable since it usually simplifies those proofs and might give some hints to study other combinatorial categories appearing in representation stability theory \cite{SS2}.
\end{enumerate}

The purpose of this paper is to get upper bounds of homological invariants in terms of generating degrees and torsion degrees, which roughly speaking, measure the ``tops" and ``socles" of finitely generated $\FI$-modules with respect to the actions of noninvertible morphisms. As in \cite{LY}, the main technical tools we use are the shift functor $\Sigma$ and its induced cokernel functor $D$ introduced in \cite[Subsection 2.3]{CEFN} and \cite[Section 3]{CE}.

\subsection{Notation}

Throughout this paper let $\mk$ be a commutative Noetherian ring, and let $G$ be a finite group. By $\C$ we denote the full subcategory of $\FI_G$ whose objects are parameterized by $\Z$, the set of nonnegative integers. It is clear that $\C$ is equivalent to $\FI_G$. The reader may refer to \cite{GL} or \cite{SS2} for a definition of $\FI_G$. As in \cite{LY}, we let $\tC$ be the $k$-linearization of $\C$.

We briefly recall the following definitions; for details, please refer to \cite{Li} and \cite{LY}.

A \textit{representation} $V$ of $\C$, or a $\C$-\emph{module} $V$, is a covariant functor $V$ from $\C$ to the category of left $\mk$-modules. Equivalently, $V$ is a $\tC$-module, a $\mk$-linear covariant functor. A $\C$-module $V$ is \emph{finitely generated} if there exists a surjective homomorphism
\begin{equation*}
\bigoplus _{i \in \Z} \tC(i, -) ^{\oplus a_i} \to V
\end{equation*}
such that $\sum _{i \in \Z} a_i < \infty$. It is \emph{torsion} if its value $V_i$ on $i$ is 0 for $i \gg 0$.

We define
\begin{align*}
J & = \bigoplus _{0 \leqslant i < j} \tC(i, j)\\
\tC_0 & = \bigoplus _{i \in \Z} \tC(i, i) \cong \tC / J,
\end{align*}
which both are $(\tC, \tC)$-bimodules. The \emph{torsion degree} of $V$ is
\begin{equation*}
\td(V) = \sup \{ i \in \Z \mid \Hom _{\tC} (\tC(i, i), V) \neq 0 \}
\end{equation*}
or $-\infty$ if the above set is empty, and in the latter case we say that $V$ is \emph{torsionless}. For $s \geqslant 0$, the $s$-th \emph{homology} and \emph{homological degree} are
\begin{align*}
H_s (V) & = \Tor _s^{\tC} ( \tC_0, V)\\
\hd_s(V) & = \td(H_s(V)).
\end{align*}
The $0$-th homological degree is called the \emph{generating degree}, and is denoted by $\gd(V)$.

\begin{remark} \normalfont
From the viewpoint of representation theory, for a finitely generated $\C$-module $V$, the zeroth homology $H_0(V)$ is precisely the top of $V$ with respect to the actions of non-invertible morphisms in $\C$. Dually, $\Hom_{\C} (\tC_0, V)$ is the socle of $V$ with respect to the actions of non-invertible morphisms in $\C$, which is nothing but the following set:
\begin{equation*}
\bigoplus _{n \geqslant 0} \{ v \in V_n \mid \alpha \cdot v = 0 \quad \forall \alpha \in \C(n, n+1) \}.
\end{equation*}
We remind the reader that although $\tC_0$ is an infinite direct sum, $\Hom_{\C} (\tC_0, V)$ is still a direct sum, and is actually a finite direct sum by the locally Noetherian property of $\tC$. Correspondingly, the generating degree and torsion degree of $V$ measure the top and socle of $V$ with respect to actions of non-invertible morphisms in $\C$ respectively.
\end{remark}

A finitely generated $\tC$-module $V$ is a \emph{basic filtered module} if there is a certain $n \in \Z$ such that $V \cong \tC \otimes _{\mk G_n} V_n$, where by $G_n$ we denote $\C(n, n)$, the group of endomorphisms of object $n$. The module $V$ is \emph{filtered} if it has a filtration by basic filtered modules.

\begin{remark} \normalfont
Filtered modules were introduced by Nagpal in \cite{N} as natural generalizations of projective modules. They have been shown to have similar behaviors as projective modules in many aspects, and hence are very useful for homological computations. Homological characterizations of these objects are described in \cite[Theorem 1.3]{LY}, which were independently obtained by Ramos in \cite{R} via a different approach.
\end{remark}

\subsection{Main result}

The main result of this paper is:

\begin{theorem}[Upper bounds for homological invariants] \label{main result}
Let $\mk$ be a commutative Noetherian ring and let $V$ be a finitely generated $\C$-module. Then we have:
\begin{enumerate}
\item $\td(V) \leqslant \gd(V) + \hd_1(V) - 1$.

\item For $s \geqslant 1$, $\hd_s(V) \leqslant \max \{ 2\gd(V)-1, \, \td(V) \} + s$ and $\hd_s(V) \leqslant \gd(V) + \hd_1(V) + s - 1$.

\item If the projective dimension $\pd(V) < \infty$, then $V$ is filtered and $\pd(V) \leqslant \fdim \mk$, the finitistic dimension of $\mk$.

\item There exists a complex
\begin{equation*}
F^{\bullet}: \quad 0 \to V \to F^{-1} \to F^{-2} \to \ldots \to F^{-n-1} \to 0
\end{equation*}
such that
\begin{itemize}
\item each $F^i$ is a filtered module;

\item $n \leqslant \gd(V)$;

\item all homologies of $F^{\bullet}$ are finitely generated torsion modules, and
\begin{equation*}
\begin{cases}
\td(H_i(F^{\bullet})) = \td(V), & \text{for }i = -1;\\
\td(H_i(F^{\bullet})) \leqslant 2\gd(V) + 2i + 2, & \text{for } -n - 2 \leqslant i \leqslant -2.
\end{cases}
\end{equation*}
\end{itemize}

\item $\Sigma_n V$ is a filtered module for $n > \max \{\td(V), \, 2\gd(V) -2 \}$.

\item If $\mk$ is a field, then there exists a rational polynomial $f$ such that $\dim_k (V_n) = f(n)$ for $n > \max \{\td(V), \, 2\gd(V) -2 \}$.

\item The injective dimension $\id(V) \leqslant \max \{ 2\gd(V) - 1, \, \td(V) \}$ whenever $\mk$ is a field of characteristic 0 and $V$ is not injective.
\end{enumerate}
\end{theorem}

\begin{remark}
The upper bounds of torsion degree and Castelnuovo-Mumford regularity in terms of the first two homological degrees were firstly proved for $\FI$ by Church and Ellenberg in \cite[Theorem 3.8 and Theorem A]{CE}, where in Theorem 3.8 of that paper we can let $a = p =1$. This upper bound of Castelnuovo-Mumford regularity was generalized to $\FI_G$ by Ramos using essentially the same idea in \cite[Theorem A]{R}. The upper bound of Castelnuovo-Mumford regularity in terms of the generating degree and torsion degree was described in \cite[Theorem 1.4]{LY}, whose proof relied on \cite[Theorem A]{CE}. In this paper we  give a proof without going through the combinatorial structure of $\FI_G$.

Statement (3) is a part of \cite[Theorem 1.3]{LY}. By definition, the \emph{finitistic dimension} of $\mk$ is the supremum of projective dimensions of finitely generated $\mk$-modules whose projective dimension is finite. By the Auslander-Buchsbaum formula, it coincides with the supremum of the depth of the localizations $\mk_\mathfrak{m}$ over all maximal ideals $\mathfrak{m}$.

The existence of such a complex of filtered modules described in Statement (4) for every finitely generated $\FI$-module was first established in \cite[Theorem A]{N}. In \cite{LY} we proved it via a different approach. Although it was known that all homologies of this complex are torsion, upper bounds of their torsion degrees were not obtained.

Nagpal proved in \cite[Theorem A]{N} that $\Sigma_N V$ is filtered for $N \gg 0$. Actually, when $\mk$ is a field of characteristic 0, the same conclusion has been established in \cite{GL}. In \cite[Theorem C]{R} Ramos gave a sharp lower bound in terms of \emph{derived regularity}, which is precisely the maximum of torsion degrees of those homologies in the complex $F^{\bullet}$. Therefore, Statement (5) actually provides an upper bound for the derived regularity of $V$.

An important representation stability phenomenon of finitely generated $\FI$-modules is the so called \emph{polynomial growth property}. This was shown for fields of characteristic 0 in \cite{CEF} and for arbitrary fields in \cite{CEFN} using a different method. That is, for $N \gg 0$, dimensions of $V_N$ satisfy a rational polynomial. In \cite[Theorem D]{R} a lower bound of such $N$ was described in terms of the generating degree and the relation degree.

When $\mk$ is a field of characteristic 0, every finitely generated projective module is injective as well. Moreover, every finitely generated module has finite injective dimension. This fact first appeared in \cite{SS1} in the language of twisted commutative algebras, and was generalized to $\FI_G$ by Gan and the author in \cite{GL} by considering the coinduction functor, which is the right adjoint of the shift functor. But upper bounds for injective dimensions were not described explicitly.
\end{remark}

\section{Castelnuovo-Mumford regularity}

Let $V$ be a finitely generated $\C$-module. The main goal of this section is to find a relationship between the torsion degree of $V$ and its first two homological degrees, and prove the two upper bounds of Castelnuovo-Mumford regularity.

Recall that there is a self-embedding functor $\iota: \C \to \C$ which sends an object $i \in \Z$ to $i+1$, and it induces the shift functor $\Sigma$. This gives a natural map $V \to \Sigma V$, whose cokernel is denoted by $DV$. Therefore, we obtain the following exact sequence
\begin{equation*}
0 \to K \to V \to \Sigma V \to DV \to 0,
\end{equation*}
where $K$ is the kernel.

\begin{lemma} \label{the kernel}
If $V$ is generated in degree 0, then $DV = 0$. Otherwise, $\gd(DV) = \gd(V) -1$. Moreover, the kernel $K$ has the following description \footnote{This description of the kernel was also mentioned by Ramos in a personal communication with the author.}
\begin{equation*}
K \cong \Hom_{\C} (\tC_0, V) \cong \bigoplus _{n \geqslant 0} \{ v \in V_n \mid \alpha \cdot v = 0 \, \forall \alpha \in \C(n, n+1) \}.
\end{equation*}
\end{lemma}

Note that although $\tC_0$ is an infinitely generated $\C$-module, $\Hom_{\C} (\tC_0, V)$ is finitely generated. Indeed, since $V$ is finitely generated, and $\tC$ is a locally Noetherian category, for $n \gg 0$, we have $\Hom_{\C} (\tC(n, n), V) = 0$.

\begin{proof}
The first statement was proved in \cite[Proposition 2.4]{LY}. To show the second statement, we just observe that the natural map $\iota^{\ast}: V \to \Sigma V$ consists of a family of maps $\iota^{\ast}_n: V_n \to (\Sigma V)_n = V_{n+1}$, which is induced by the inclusion $\iota_n: [n] \to [n+1]$ by adding 1 to every number in $[n]$. Furthermore, for $v \in V_n$, if its image in $V_{n+1}$ under a morphism $\alpha \in \C(n, n+1)$ is not 0, then its images under all morphisms in $\C(n, n+1)$ are nonzero since the group $\C(n+1, n+1)$ acts transitively on $\C(n, n+1)$, and in particular the image of $v$ under $\iota_n^{\ast}$ is not 0 as well. This shows that $K$ consists of all elements in $V$ such that there exists a certain $n \in \Z$ with $\C(n, n+1) \cdot v = 0$, and the conclusion follows.
\end{proof}

\begin{remark} \normalfont
In \cite{CE} Church and Ellenberg defined the left derived functor $H^D_{\bullet}$ of the functor $D$, and this was also used by Ramos in \cite{R}. The kernel $K$ appearing in this exact sequence is precisely $H^D_1(V)$. Therefore, the first derived functor $H^D_1$ of $D$ is isomorphic to the classical functor $\Hom _{\C} (\tC_0, -)$.
\end{remark}

An immediate corollary is:

\begin{corollary}
Let $K$ be as in the previous lemma. Then $\td(V) = \td(K) = \gd(K)$.
\end{corollary}

Now we can give another proof for the upper bound of Castelnuovo-Mumford regularity described in \cite[Theorem A]{CE} and \cite[Theorem A]{R}.

\begin{theorem}
Let $V$ be a finitely generated $\C$-module. Then
\begin{equation*}
\td(V) \leqslant \gd(V) + \hd_1(V) - 1
\end{equation*}
and for $s \geqslant 1$,
\begin{equation*}
\hd_s(V) \leqslant \gd(V) + \hd_1(V) + s - 1.
\end{equation*}
\end{theorem}

\begin{proof}
We use induction on $\gd(V)$. The conclusion holds trivially for $V = 0$. Now suppose that $V \neq 0$. By \cite[Corollary 3.4]{LY}, there is a short exact sequence $0 \to V' \to V \to V'' \to 0$ such that $V''$ is filtered and $\hd_1(V') > \gd(V')$. Since $V''$ is torsionless by \cite[Proposition 3.9]{LY}, and $H_s(V'') = 0$ for $s \geqslant 1$ by \cite[Theorem 1.3]{LY}, we know that $\td(V) = \td(V')$ and $\hd_s(V) = \hd_s(V')$ for $s \geqslant 1$. Moreover, the following short exact sequence $0 \to H_0(V') \to H_0(V) \to H_0(V'') \to 0$ implies that $\gd(V) \geqslant \gd(V')$. Therefore, it suffices to show the inequalities for $V'$.

If $\gd(V') < \gd(V)$, then the conclusion for $V'$ follows from the induction hypothesis. Otherwise, $\gd(V') = \gd(V)$. In that case we consider the following commutative diagram
\begin{equation*}
\xymatrix{
0 \ar[r] & W \ar[r] \ar[d] & \Sigma W \ar[r] \ar[d] & DW \ar[d]^{\delta} \ar[r] & 0\\
0 \ar[r] & P \ar[r] \ar[d] & \Sigma P \ar[r] \ar[d] & DP \ar[r] \ar[d] & 0\\
K \ar[r] & V' \ar[r] & \Sigma V' \ar[r] & DV' \ar[r] & 0,
}
\end{equation*}
where the first two rows and columns are short exact sequences, and $P$ is a projective module satisfying $\gd(P) = \gd(V')$.  To construct it, one starts from the injective map $W \to P$ to get the first two rows. Applying the snake lemma to them, one gets an exact sequence
\begin{equation*}
0 \to \ker \delta \to V' \to \Sigma V' \to DV' \to 0.
\end{equation*}
But it is known that the kernel of $V' \to \Sigma V'$ is $K$. Putting these pieces of information together, we get the above diagram. We remind the reader that in the third column the map $DP \to DV'$ is surjective. Furthermore, the third column gives us two exact sequences
\begin{align*}
& 0 \to K \cong \ker \delta \to DW \to U \to 0,\\
& 0 \to U \to DP \to DV' \to 0.
\end{align*}
Note that $\gd(DV') < \gd(V') = \gd(V)$. Moreover, since $\hd_1(V') > \gd(V')$, by \cite[Lemma 2.9]{LY}, one has $\hd_1(V') = \gd(W)$.

The first sequence induces a long exact sequence
\begin{equation*}
\ldots \to H_1(U) \to H_0(K) \to H_0(DW) \to H_0(U) \to 0,
\end{equation*}
so we get
\begin{align*}
\td(V') & = \td(K) = \gd(K) & \text{ by the previous corollary}\\
& \leqslant \max \{ \gd(DW), \, \hd_1(U) \} & \text{ by the long exact sequence}\\
& \leqslant \max \{ \gd(W) - 1, \, \hd_2(DV') \} & \text{ by Lemma \ref{the kernel} and the second sequence}\\
& = \max \{\hd_1(V') - 1, \, \hd_2(DV') \}.
\end{align*}
By the induction hypothesis,
\begin{align*}
\hd_2(DV') & \leqslant \gd(DV') + \hd_1(DV') + 1 \leqslant \gd(V') + \hd_1(DV') & \text{ by Lemma \ref{the kernel}}\\
& \leqslant \gd(V') + \gd(U) \leqslant \gd(V') + \gd(DW) & \text{ by the two short exact sequences}\\
& \leqslant \gd(V') + \gd(W) - 1 = \gd(V') + \hd_1(V') - 1 & \text{ by Lemma \ref{the kernel}.}\\
\end{align*}
Putting these two inequalities together, we get the claimed upper bound for $\td(V)$.

Now we turn to the upper bound of homological degrees. For $s \geqslant 1$, from this long exact sequence we also have
\begin{align*}
\hd_s (DW) & \leqslant \max \{\hd_s(K), \, \hd_s(U) \} \\
& \leqslant \max \{ \td(K) + s, \, \hd_s(U) \} & \text{ by \cite[Theorem 1.5]{Li}}\\
& = \max \{ \gd(V') + \hd_1(V') + s - 1, \, \hd_{s+1} (DV') \} & \text{ by what we just proved.}
\end{align*}
But by the induction hypothesis,
\begin{align*}
\hd_{s+1} (DV') & \leqslant \gd(DV') + \hd_1(DV') + s \leqslant \gd(V') - 1 + \gd(U) + s & \text{ by the second sequence}\\
& \leqslant \gd(V') + \gd(DW) + s - 1 & \text{ by the first sequence}\\
& \leqslant \gd(V') + \gd(W) + s - 2 = \gd(V') + \hd_1(V') + s - 2 & \text{ by Lemma \ref{the kernel}.}
\end{align*}
Combining these two inequalities, we have
\begin{equation*}
\hd_s (DW) \leqslant \gd(V') + \hd_1(V') + s - 1.
\end{equation*}
Note that $W$ is a torsionless module, and $\gd(W) = \hd_1(V') > \gd(V') \geqslant 0$. Therefore, by \cite[Corollary 2.12]{LY}, we get
\begin{align*}
\hd_{s+1} (V') & = \hd_s(W) \leqslant \max \{ \gd(DW), \, \ldots, \, \hd_s(DW) \} + 1 \leqslant \gd(V') + \hd_1(V') + s
\end{align*}
for $s \geqslant 1$. This inequality holds obviously for $s = 0$. The conclusion then follows from induction.
\end{proof}

We know that $\Sigma_N V$ is filtered for $N \gg 0$. As pointed out in \cite{LY}, this result as well as the above theorem imply another upper bound of the Castelnuovo-Mumford regularity. For the convenience of the reader, we repeat the proof.

\begin{corollary}
Let $V$ be a finitely generated $\C$-module. Then for $s \geqslant 1$,
\begin{equation*}
\hd_s(V) \leqslant \max \{ 2\gd(V) - 1, \td(V) \} + s.
\end{equation*}
\end{corollary}

\begin{proof}
There is a classical short exact sequence
\begin{equation*}
0 \to V_T \to V \to V_F \to 0,
\end{equation*}
where $V_T$ is a torsion module and $V_F$ is torsionless.

For the torsion part $V_T$, one has
\begin{equation}
\hd_s(V_T) \leqslant \td(V_T) + s = \td(V) + s
\end{equation}
by \cite[Theorem 1.5]{Li}. For $V_F$, there is a short exact sequence
\begin{equation*}
0 \to V_F \to \Sigma_N V_F \to C \to 0
\end{equation*}
where $N$ is large enough such that $\Sigma_N V_F$ is a filtered module; see \cite[Theorem 1.6]{LY}. The long exact sequence and homological characterizations of filtered modules (\cite[Theorem 1.3]{LY}) tell us that
\begin{equation*}
\hd_s(V_F) = \hd_{s+1} (C) \leqslant \gd(C) + \hd_1(C) + s
\end{equation*}
for $s \geqslant 1$. We also note that
\begin{align*}
& \gd(C) \leqslant \gd(V_F) - 1 \leqslant \gd(V) - 1,\\
& \hd_1(C) \leqslant \gd(V_F) \leqslant \gd(V).
\end{align*}
Putting these pieces of information together, we have
\begin{equation}
\hd_s(V_F) \leqslant 2\gd(V) + s -1.
\end{equation}
The conclusion then follows from (2.1) and (2.2).
\end{proof}

\section{Complexes of filtered modules}

In this section we construct a finite complex of filtered modules for each finitely generated $\C$-module $V$, and use it to prove the last four statements in Theorem \ref{main result}. We will see from the proofs that the stability phenomena of finitely generated $\C$-modules actually follow from those of filtered modules.

Recall that there is a classical short exact sequence
\begin{equation*}
0 \to V_T \to V \to V_F \to 0,
\end{equation*}
and $\Sigma_N V_F$ is filtered for $N \gg 0$. Choose a large enough $N$ and denote $\Sigma_N V_F$ by $F^{-1}$. We get a map $V \to F^{-1}$ which is the composite of $V \to V_F$ and $V_F \to F^{-1}$. Let $V^{-1}$ be the cokernel of this map, and repeat the above process for $V^{-1}$. Eventually we construct a complex of filtered modules. This is a finite complex with torsion homologies. Explicitly, in the complex
\begin{equation*}
\xymatrix{
F^{\bullet}: & 0 \ar[r] ^-{\delta^0} & V \ar[r] ^-{\delta^{-1}} & F^{-1} \ar[r] ^-{\delta^{-2}} & F^{-2} \ar[r] ^{\delta^{-3}} & \ldots \ar[r] ^-{\delta^{-n-1}} & F^{-n-1} \ar[r] ^-{\delta^{-n-2}} & 0.
}
\end{equation*}
one has $n \leqslant \gd(V)$. Moreover, if we let $V^0 = V$ and let $V^i$ be the cokernel of $\delta^i$ for $-n - 2 \leqslant i \leqslant -1$, then from the above construction one easily sees that the image of $\delta^i$ is $V^{i+1}_F$, the torsionless part of $V^{i+1}$, and the $i$-th homology
\begin{equation*}
H_i (F^{\bullet}) = \ker \delta^i / \Image \delta ^{i+1} \cong V^{i+1}_T,
\end{equation*}
the torsion part of $V^{i+1}$. For details, see the proof of \cite[Theorem A]{N} or the proof of \cite[Theorem 4.4]{LY}.

\begin{proposition}
In the above complex, one has
\begin{equation*}
\td(H_{-1} (F^{\bullet})) = \td(V)
\end{equation*}
and for $-n - 2 \leqslant i \leqslant -2$,
\begin{equation*}
\td(H_i(F^{\bullet})) \leqslant 2\gd(V) + 2i + 2.
\end{equation*}
\end{proposition}

\begin{proof}
The conclusion holds obviously for $i = -1$ since $H_{-1} (F^{\bullet}) \cong V_T$. For $-n - 2 \leqslant i \leqslant -2$, one has $H_i (F^{\bullet}) \cong V^{i+1}_T$, and the following short exact sequence
\begin{equation*}
0 \to V_F^{i+2} \to F^{i+1} \to V^{i+1} \to 0
\end{equation*}
induces an exact sequence
\begin{equation*}
0 \to H_1(V^{i+1}) \to H_0 (V_F^{i+2}) \to H_0(F^{i+1}) \to H_0(V^{i+1}) \to 0.
\end{equation*}
By the first statement in Theorem \ref{main result}, one has
\begin{align*}
\td (H_i (F^{\bullet})) & = \td (V_T^{i+1}) = \td (V^{i+1})\\
& \leqslant \gd(V^{i+1}) + \hd_1(V^{i+1}) - 1\\
& \leqslant \gd(V^{i+1}) + \gd(V_F^{i+2}) - 1\\
& \leqslant \gd(V^{i+1}) + \gd(V^{i+2}) - 1\\
& \leqslant (\gd(V) + i + 1) + (\gd(V) + i + 2) - 1\\
& = 2\gd(V) + 2i + 2,
\end{align*}
where the last inequality follows from the proof of \cite[Theorem 4.14]{LY}.
\end{proof}

\begin{remark}
In \cite{R} Ramos defined the notion \emph{derived regularity} for finitely generated modules $V$. Actually, we can see that the derived regularity of $V$ is precisely the maximum of torsion degrees of homologies in this complex. From this observation, one immediately deduces that the derived regularity of $V$ is bounded by $\max \{\td(V), \, 2\gd(V) - 2\}$.
\end{remark}

We get three corollaries:

\begin{corollary}
Let $V$ be a finitely generated $\C$-module. Then for $N > \max \{\td(V), \, 2\gd(V) - 2 \}$, $\Sigma_N V$ is a filtered module.
\end{corollary}

\begin{proof}
Apply the functor $\Sigma_N$ to the finite complex of filtered modules. Note that $\Sigma_N$ is an exact functor and it preserves filtered modules by \cite[Proposition 3.9]{LY}. Moreover, all homologies vanish since their torsion degrees are strictly less than $N$. Consequently, we get a resolution of $\Sigma_N V$ by filtered modules. However, by \cite[Corollary 4.3]{LY}, we deduce that $\Sigma_N V$ is filtered as well.
\end{proof}

\begin{remark} \normalfont
In \cite[Theorem C]{R} Ramos showed that if $N$ is strictly greater than the derived regularity of $V$, then $\Sigma_N V$ is filtered. Since we already know that the derived regularity is bounded by $\max \{ \td(V), \, 2\gd(V) - 2 \}$, the above corollary follows from this fact immediately.
\end{remark}

\begin{corollary}
Let $\mk$ be a field and let $V$ be a finitely generated $\C$-module. Then there exists a rational polynomial $f \in \mathbb{Q}[X]$ such that $\dim_{\mk} V_N = f(N)$ for $N > \max \{\td(V), \, 2\gd(V) - 2 \}$.
\end{corollary}

\begin{proof}
Note that in the complex
\begin{equation*}
0 \to V \to F^{-1} \to \ldots \to F^{-n-1} \to 0,
\end{equation*}
the dimensions $\dim_{\mk} F^i_s$ satisfy polynomials $f_i \in \mathbb{Q}[X]$ for $-n-1 \leqslant i \leqslant -1$ and $s \geqslant \gd(F^i)$. Since $\gd(F^i) \leqslant \gd(V)$ by the proof of \cite[Theorem 4.14]{LY}, this polynomial growth phenomenon holds for all $F^i$ when $s > \max \{\td(V), \, 2\gd(V) - 2 \}$, including the cases that $\gd(V) = 0$ or $\gd(V) = 1$. Therefore, for $N > \max \{\td(V), \, 2\gd(V) - 2 \}$ and an arbitrary $i$, all $\dim_{\mk} F^i_N$ satisfy the polynomial $f_i$. Moreover, we also know that all homologies are supported on objects $\leqslant \max \{\td(V), \, 2\gd(V) - 2 \}$. Consequently, for $N > \max \{\td(V), \, 2\gd(V) - 2 \}$, the dimensions $\dim_{\mk} V_N$ satisfy a polynomial, which is a linear combination of those $f_i$.
\end{proof}

\begin{remark} \normalfont
The minimal number $N$ such that for $n \geqslant N$, all $\dim_{\mk} V_n$ satisfy a polynomial is called the \emph{stable range} of $V$ by Ramos in \cite{R}. Theorem D in his paper asserts that the stable range of $V$ is bounded by $r + \min \{r, \, \gd(V)\}$, where $r$ is the relation degree. Since $\td(V) \leqslant \gd(V) + \hd_1(V) -1$ and $\hd_1(V) \leqslant r$, for general modules the bound provided in the above corollary might be a little bit more optimal.
\end{remark}

\begin{corollary}
The injective dimension $\id(V) \leqslant \max \{ 2\gd(V) - 1, \, \td(V) \}$ whenever $\mk$ is a field of characteristic 0 and $V$ is not injective.
\end{corollary}

\begin{proof}
We claim that $\max \{ 2\gd(V) - 1, \, \td(V) \} \geqslant 0$. Otherwise, $V$ is 0 or a direct sum of $\tC(0, -)$, contradicting the assumption that $V$ is not injective.

We prove the conclusion by induction on $\gd(V)$. If $\gd(V) = 0$, then the short exact sequence
\begin{equation*}
0 \to V_T \to V \to V_F \to 0
\end{equation*}
and the given assumption imply that $V \cong V_T \oplus V_F$ since $V_F$, if it is nonzero, is a torsionless module generated in degree 0, which must be projective. Therefore, $\id(V) = \id(V_T)$ since $V_F$ is injective as well. But it is well known that for a finite dimensional module $V_T$, its injective dimension is bounded by $\td(V_T) = \td(V)$, so the conclusion follows. To see this, one just observes that the value of $V_T$ on the object $\td(V_T)$ is in the socle of $V_T$. Therefore, there is an injective homomorphism $V_T \to I$ such that $I$ is a finite dimensional injective module and their values on the object $\td(V_T)$ are isomorphic. Consequently, the torsion degree of the cokernel of this map is strictly less than that of $V_T$. Now one can proceed by recursion.

For a general module $V$, consider the short exact sequence $0 \to V_T \to V \to V_F \to 0$. For $V_T$, we have $\id(V_T) \leqslant \td(V_T) = \td(V)$. For $V_F$, we have a short exact sequence
\begin{equation*}
0 \to V_F \to P \to W \to 0
\end{equation*}
where $P$ is a projective module, which is injective as well. Moreover, by Proposition 7.5 and Theorem 1.7 in \cite{GL}, one can assume that $\gd(P) < \gd(V_F)$, and $W$ has no projective summands. Consequently,
\begin{equation*}
\gd(W) \leqslant \gd(P) < \gd(V_F) \leqslant \gd(V)
\end{equation*}
and
\begin{equation*}
\td(W) \leqslant \gd(W) + \hd_1(W) - 1 \leqslant \gd(V) - 1 + \gd(V_F) - 1 \leqslant 2\gd(V) - 2
\end{equation*}
since $H_1(W) \subseteq H_0(V_F)$. By induction hypothesis,
\begin{equation*}
\id(W) \leqslant \max \{ 2\gd(W) - 1, \, \td(W) \} \leqslant \max \{ 2\gd(V) - 3, \, 2\gd(V) - 2 \} = 2\gd(V) - 2.
\end{equation*}
Consequently, $\id(V_F) \leqslant 2\gd(V) - 1$. Putting the two estimations together, we have the wanted conclusion for $V$.

\end{proof}


\begin{thebibliography}{99}
\bibitem{CE} T. Church, J. Ellenberg, \textit{Homology of FI-modules}, arXiv:1506.01022.
\bibitem{CEF} T. Church, J. Ellenberg, B. Farb, \textit{FI-modules and stability for representations of symmetric groups}, Duke Math. J. 164 (2015) 9, 1833-1910, arXiv:1204.4533.
\bibitem{CEFN}  T. Church, J. Ellenberg, B. Farb, R. Nagpal, \textit{$\mathrm{FI}$-modules over Noetherian rings}, Geom. Topol. 18-5 (2014) 2951-2984, arXiv:1210.1854.
\bibitem{GL} W. L. Gan, L. Li, \textit{Coinduction functor in representation stability theory}, J. Lond. Math. Soc. 92 (2015), 689-711.
\bibitem{Li} L. Li, \textit{Homological degrees of representations of categories with shift functors}, arXiv:1507.08023.
\bibitem{LY} L. Li, N. Yu, \textit{Filtrations and Homological degrees of FI-modules}, arXiv:1511.02977.
\bibitem{N} R. Nagpal, \textit{FI-modules and the cohomology of modular representations of symmetric groups}, arXiv:1505.04294.
\bibitem{R} E. Ramos, \textit{Homological invariants of FI-modules and $FI_G$-modules}, arXiv:1511.03964.
\bibitem{SS1} S. Sam, A. Snowden, \textit{GL-equivariant modules over polynomial rings in infinitely many variables}, to appear in Trans. Amer. Math. Soc., arXiv:1206.2233.
\bibitem{SS2} S. Sam, A. Snowden, \textit{Gr\"{o}bner methods for representations of combinatorial categories}, arXiv:1409.1670.
\end{thebibliography}
\end{document}